\documentclass{amsart}
\usepackage{amssymb}
\usepackage{amsmath}
\usepackage{amsfonts}

\newtheorem{theorem}{Theorem}[section]
\theoremstyle{plain}

\newtheorem{corollary}[theorem]{Corollary}

\newtheorem{example}[theorem]{Example}

\newtheorem{lemma}[theorem]{Lemma}

\newtheorem{proposition}[theorem]{Proposition}

\numberwithin{equation}{section}

\begin{document}
\title{On complemented non-abelian chief factors of a Lie algebra}
\author{David A. TOWERS}
\address{Lancaster University\\
Department of Mathematics and Statistics \\
LA$1$ $4$YF Lancaster\\
ENGLAND}
\email{d.towers@lancaster.ac.uk}
\author{Zekiye CILOGLU}
\address{Suleyman Demirel University\\
Department of Mathematics\\
$32260$, Isparta, TURKEY}
\email{zekiyeciloglu@sdu.edu.tr}
\thanks{$2010$ \textit{Mathematics Subject Classification.} $17$B$05$, $17$B$%
20$, $17$B$30$, $17$B$50$. }

\begin{abstract}
The number of Frattini chief factors or of chief factors which are complemented by a maximal subalgebra of a finite-dimensional Lie algebra $L$ is the same in every chief series for $L$, by \cite[Theorem 2.3]{[11]}. However, this is not the case for the number of chief factors which are simply complemented in $L$. In this paper we determine the possible variation in that number.
\end{abstract}

\maketitle

\section{\textbf{Preliminary Results}}

Throughout $L$ will be a finite-dimensional Lie algebra with product $[,]$ over a field. We say that $A$ is an {\em $L$-algebra} if it is a Lie algebra (with product denoted by juxtaposition) and there is a homomorphism $\theta : L \rightarrow$ Der $A$. Then $A$ is also an $L$-module with action $\cdot$ given by $x.a=\theta(x)(a)$ and we have $x.(a_1a_2)=(x.a_1)a_2-(x.a_2)a_1$. If $A$ is an ideal of $L$ we will consider it as an $L$-algebra in the natural way. Given such an $L$-algebra $A,$ we define the corresponding semi-direct sum $A\rtimes L$ as the set of ordered pairs, where the multiplication is given by
\begin{equation*}
 (a_{1}, x_{1})(a_{2},x_{2}) =(x_{1}.a_{2}-x_{2}.a_{1}+a_1a_2,[x_{1},x_{2}]) 
\end{equation*}  
for all $\ a_{1},\ a_{2}\in A$ and for all $\ x_{1},\ x_{2}\in $
$L$.
\par

Let $A$ and $B$ two $L$-algebras. An (algebra) isomorphism $\theta
:A\rightarrow B\ $is said to be an {\em $L$-isomorphism }if it is also an 
$L$-module isomorphism. Note that this is stronger than the definition used in \cite{[11]}, where $\theta$ is only required to be an $L$-module isomorphism. However, the results proved there apply equally to this stronger version. When such a $\theta \ $exists we write $A\cong
_{L}B.\ $ We say that $A$, $B$ are {\em $L$-equivalent,} written $A\sim
_{L}B$ if there is an isomorphism $\Phi :A\rtimes L\rightarrow B\rtimes L$
such that the following diagram commutes:
\begin{equation*}
\begin{array}{ccccccccc}
0 & \hookrightarrow & A & \hookrightarrow & A\rtimes L & \twoheadrightarrow & L
& \twoheadrightarrow & 0 \\ 
&  & {\Large \downarrow }{\small \phi } &  & {\Large \downarrow }{\small %
\Phi } &  & {\Large \parallel } &  &  \\ 
0 & \hookrightarrow & B & \hookrightarrow & B\rtimes L & \twoheadrightarrow & L
& \twoheadrightarrow & 0
\end{array}
\end{equation*}

In this case we say that the extensions $A \hookrightarrow A\rtimes L \twoheadrightarrow L$ and $B \hookrightarrow B\rtimes L \twoheadrightarrow L$ are \textit{equivalent}. It is clear that $L$-equivalence is an equivalence relation. 
\par

If $\phi : A \rightarrow B$ is an $L$-isomorphism, then putting $\Phi((a,x))=(\phi(a),x)$ defines an isomorphism $\Phi : A \rtimes L \rightarrow B \rtimes L$ making the above diagram commutative. It follows that $L$-isomorphic $L$-algebras are $L$-equivalent. However, the converse is false. For example, if $L=A \oplus B$, where $A$ and $B$ are isomorphic simple Lie algebras, then $A$ and $B$ are $L$-equivalent, but they are not $L$-isomorphic, as $C_L(A)=B$ and $C_L(B)=A$.
\par 

If $B$ is an $L$-algebra we define a {\em $1$-cocycle} of $L$ with values in $B$ to be a map $\beta \in Z^{1}( L,B) $ such that 
\[ \beta([x,y])=x.\beta(y)-y.\beta(x)+\beta(x)\beta(y).
\] Then the map $\theta: L \rightarrow$ Der $B$ given by $\theta(x)=\theta_x$ where $\theta_x(b)=\beta(x)b+x.b$ for all $x\in L$ and $b\in B$ is a homomorphism, and so we can define another $L$-module structure on $B$ by
\begin{equation*}
x\odot b=\beta( x) b +x.b.
\end{equation*}
We denote the $L$-algebra with this $L$-module structure by $B_{\beta }.$
\par

The following proposition gives us a useful criterion for two $L$-algebras to be
equivalent.

\begin{proposition}\label{p:1}
Let $A$ and $B$ be two $L$-algebras. They are $L$-equivalent if and only if
there is a $1$-cocycle $\beta \in Z^{1}( L,B) $ and an $L$-isomorphism $\phi$ from $A$ to $B_{\beta }$ (that is, $\phi (x.a)=x\odot \phi
(a)$ for all $x\in L,\ a\in A$). 
\end{proposition}
\begin{proof}
Suppose first that there is a $1$-cocyle $\beta \in Z^{1}( L,B)$ and an $L$-isomorphism $\phi: A\rightarrow B_{\beta }.$ Then, the map $\Phi :A\rtimes
L\rightarrow B\rtimes L$ given by
\begin{equation*}
\Phi (( a,x)) =(\phi ( a) +\beta ( x),x)
\end{equation*}
shows that $A$ and $B$ are $L$-equivalent. 
\par

Conversely, suppose that they are $L$-equivalent under the isomorphism $\Phi :A\rtimes
L\rightarrow B\rtimes L$ . Define $\beta :L\rightarrow B$ by $\beta ( x) = \pi_1(\Phi( (0,x)) -(0,x))$ where $\pi_1: B \rtimes L \rightarrow B : (b,x) \mapsto b$ is the projection map onto $B$. Then it is straightforward to check that $\beta \in Z^{1} L,B) $ and that $\phi $ is an $L$-isomorphism from $A$ to $B_{\beta }.$
\end{proof}

If $A$ and $B$ are abelian and $L$-equivalent, they have the same dimension, and so are $L$-isomorphic. However, as we have seen, for nonabelian $L$-algebras, $L$-equivalence is strictly weaker than $L$-isomorphism.
\par

Recall that the factor algebra $A/B$ is called a {\em chief factor} of $L$ if $B$ is an ideal of $L$ and $A/B$ is a minimal ideal of $L/B$. The  {\em Frattini ideal} of $L$, denoted by $\phi(L)$, is the largest ideal of $L$ contained in the intersection of all of the maximal subalgebras of $L$. A chief factor $A/B$ is called {\em Frattini }if $A/B\subseteq \phi
\left( L/B\right) .$

If there is a subalgebra, $M$ such that $L=A+M$ and $B\subseteq A\cap M,$ we
say that $A/B\ $is$\ $a \textit{supplemented} chief factor of $L,$ and that $%
M$ is a \textit{supplement} of $A/B\ $in $L.$ Also, if $A/B$ is a
non-Frattini chief factor of $L$, then $A/B$ is \textit{supplemented} by a
maximal subalgebra $M$ of $L.$

If $A/B\ $is$\ $a chief factor of $L$ supplemented by a subalgebra $M$ of $%
L, $ and \linebreak $A\cap M=B$ then we say that $A/B$ is \textit{%
complemented} chief factor of $L,$ and $M$ is a \textit{complement} of $A/B$
in $L$. When $L$ is \textit{solvable}, it is easy to see that a chief factor
is Frattini  if and only if it \noindent is not complemented.

The {\em centralizer} of an $L$-algebra $A$ in $L$ is $C_{L}( A)=\{ x\in L\mid x.a=0$ for all $a\in A\} .$ We will denote an algebra direct sum by `$\oplus$', whereas `$\dot{+}$' will denote a direct sum of the underlying vector space only. Then the following proposition gives a criterion for a nonabelian chief factor to be complemented.

\begin{proposition}\label{p:2}
Let $A_{1}/B_{1}$ be a nonabelian chief factor of $L.$ Then, $A_{1}/B_{1}$ is complemented in $L$ if and only if there exists an $L$-algebra $B$ such that$\ A_{1}/B_{1}\sim _{L}B,$ and $A_{1}\subseteq
C_{L}( B)$.
\end{proposition}

\begin{proof}
$( \Longrightarrow) \ $Suppose that $A_{1}/B_{1}$
is complemented in $L$ and $M$ is a complement of $A_{1}/B_{1}$ in $L$. Since $L=M+A_{1},$ for each $x\in L$ we can write $x=m_x+a_x$ for some $m_x\in M$ and $a_x\in A_{1}.$ We
consider the $L$-algebra $B$ whose underlying algebra is $A_{1}/B_{1}$ with
the module operation:
\begin{equation*}
\begin{array}{ccccc}
\wedge & : & L\times B & \rightarrow & B \\ 
&  & ( x,b) & \rightarrow & [m_x,a_b]+B_1
\end{array}
\end{equation*}
where $b=a_b+B_1$ ($a_b \in A_1$).
Define the $1$-cocycle $\beta \in Z^{1}( L,B) $ as;
\begin{equation*}
\begin{array}{lllll}
\beta & : & L & \rightarrow & B \\ 
&  & x & \rightarrow & \beta \left( x\right) =a_x+B_{1}\ \left( a_x\in A_{1}\right)
\end{array}
\end{equation*}
 It is immediate that both are well defined mappings and that $\beta$ is a $1$-cocycle. Let 
$\phi :A_{1}/B_{1}\rightarrow B$ be given by, $\phi \left( a_{1}+B_{1}\right)
=a_{1}+B_{1}\ $for all $a_{1}+B_{1}\in A_{1}/B_{1}$. Then we
can define another module structure on $B$ using $\beta $ and, for all $x\in L$ and for all $a_b+B_{1}\in A_{1}/B_{1}$, we have
\begin{equation*}
\begin{array}{llll}
\phi ( [x,a_b+B_1]) & = &  \phi([x,a_b]+B_{1}) &  \\ 
& = & [x,a_b]+B_{1} &  \\ 
& = & [a_x+B_1,a_b+B_1]+[m_x,a_b+B_1] \\ 
& = & [\beta(x),a_b+B_1]+x \wedge (a_b+B_1) \\
& = & x \odot (a_b+B_1). & 
\end{array}
\end{equation*}
Hence $\phi $ is an $L$-isomorphism
and $A_{1}/B_{1}\cong _{L}B_{\beta }.$ Then, using Proposition \ref{p:1}, we have that $A_{1}/B_{1}\sim _{L}B.$  Also 
\begin{equation*}
\begin{array}{lll}
C_{L}( B) & = & \{ x\in L\mid x\wedge B
=0_B\} \\ 
& = & \{ m_x+a_x\in L\mid [ m_x,a_b]+B_1 =B_1  \hbox{ for all } b \in B \} \\ 
& = & C_{M}\left( A_{1}/B_{1}\right) \oplus A_{1}%
\end{array}
\end{equation*}
whence $A_{1}\subseteq C_{L}\left( B\right) .$ 

$( \Longleftarrow) $ Assume now that $B$ is an $L$-algebra, $A_{1}/B_{1}\sim _{L}B$ and $A_{1}\subseteq C_{L}( B) .$ We need to show that $A_{1}/B_{1}$ is complemented in $L$. Since $A_{1}/B_{1}\sim _{L}B$ we have
an $L$-isomorphism $\phi : B \rightarrow (A_1/B_1)_{\alpha}$ where $\alpha \in Z^{1}(L,A_{1}/B_{1}) $, by Proposition \ref{p:1}, and $A_{1}\subseteq C_{L}( B)$. 
 If $b \in B$, then
\[
\phi(b) = \phi(b + a_1.b)= \phi(b) + a_1 \odot \phi(b)= \phi(b)+[\alpha(a_1),\phi(b)]+[a_1+B_1,\phi(b)]
\]
so $[\alpha(a_1)+a_1+B_1,\phi(b)]=B_1$ for all $b \in B$; that is, $$\alpha(a_1)+a_1+B_1 \in C_L(A_1/B_1)\cap A_1/B_1=B_1,$$ since $A_1/B_1$ is a nonabelian chief factor of $L$. Hence $\alpha(a_1)=-a_1+B_1$.
\par

Put $M=Ker( \alpha)$. Let $x\in L$ and $\alpha ( x) =a_{1}+B_{1}.$ Then
\[
\alpha(x+a_1)=\alpha(x)+ \alpha(a_1)=(a_1+B_1)+(-a_1+B_1)=B_1
\]
so $x+a_1 \in M$. Hence $L=M+A_1$. If $m \in M \cap A_1$ we have $B_1= \alpha(m)=-m+B_1$, so $M \cap A_1=B_1$
and $M$ is a complement of $A_{1}/B_{1}\ $ in $L$.
\end{proof}

Recall that, 
\begin{itemize}
\item[(i)] the {\em socle} of $L,$ $Soc\left( L\right) $ is
the sum of all of the minimal non-zero ideals of $L;$ and
\item[(ii)] if $U$ is a subalgebra of $L,$ the {\em core} of $U$,  $U_{L}$,
is the largest ideal of $L$ contained in $U$. We say
that $U$ is {\em core-free} in $L$ if $U_{L}=0.$ 
\end{itemize}

We shall call $L$ {\em primitive} if it has a core-free maximal subalgebra. Then we have the following characterisation of primitive Lie algebras.

\begin{theorem}\label{t:1} (\cite[Theorem 1.1]{[11]})
\begin{itemize}
\item[(i)]  A Lie algebra $L$ is primitive if and only if there exists a subalgebra $M$ of $L$ such
that $L = M+A$ for all minimal ideals $A$ of $L$.
\item[(ii)]  Let $L$ be a primitive Lie algebra. Assume that $U$ is a core-free maximal subalgebra
of $L$ and that $A$ is a non-trivial ideal of $L$. Write $C = C_L(A)$.
Then $C \cap U = 0$. Moreover, either $C = 0$ or $C$ is a minimal ideal of $L$.
\item[(iii)]  If L is a primitive Lie algebra and $U$ is a core-free maximal subalgebra of $L$,
then exactly one of the following statements holds:
\begin{itemize}
\item[(a)] $Soc(L) = A$ is a self-centralising abelian minimal ideal of $L$ which is complemented by $U$; that is, $L = U \dot{+} A$.
\item[(b)] $Soc(L) = A$ is a non-abelian minimal ideal of $L$ which is
supplemented by $U$; that is $L = U+A$. In this case $C_L(A) = 0$.
\item[(c)] $Soc(L) = A \oplus B$, where $A$ and $B$ are the two unique minimal ideals of $L$ and both are complemented by $U$; that is, $L = A \dot{+} U = B \dot{+} U$. In this case $A = C_L(B)$, $B = C_L(A)$, and $A$, $B$ and $(A+B) \cap U$ are nonabelian isomorphic algebras.
\end{itemize}
\end{itemize}
\end{theorem}
We say that $L$ is  
\begin{itemize}
\item {\em primitive of type $1$} if it has a unique minimal ideal that is abelian;
\item {\em primitive of type $2$} if it has a unique minimal ideal that is  non-abelian; and
\item {\em primitive of type $3$} if it has precisely two distinct minimal ideals each of which is  non-abelian.
\end{itemize}

Let $A/B$ and $D/E$ be chief factors of $L.$ We say that they are  
{\em $L$-connected}, if either they are $L$-isomorphic or there exists an
epimorphic image of $L$ which is primitive of type $3$ and whose
minimal ideals are $L$-isomorphic to the given factors. The property of being $L$-connected is an equivalence relation on the set of chief factors.
The set of chief factors of $L$ is denoted as:
\begin{equation*}
C\mathcal{F}(L)=\{ A/B \mid A,B \text{ are ideals of } L,A/B \text{ is a chief
factor of }L\}.
\end{equation*}
Let 
\begin{equation*}
I_{L}( A) =\{ x\in L \mid \text{ad}\,x\mid_A
=\text{ad}\,a\text{ for some }a\in A\},
\end{equation*}
where $A$ is an $L$-algebra (and $ \text{ad}\,x\mid_A$ refers to the module action of $x$ on $A$.) 

\begin{lemma}\label{l:1} 
\begin{itemize}  
\item[(i)] Let $A,B$ be ideals of a Lie algebra $L$ with $B \subseteq A$. Then $I_L(A/B)=A+C_L(A/B)$.
\item[(ii)] Let $A$ be an $L$-algebra with $C_L(A) \subseteq I_L(A)$. Then $I_L(A)/C_L(A)$ is isomorphic to a subalgebra of $A/Z(A)$.
\item[(iii)] $A$ is an abelian $L$-algebra if and only if $I_L(A)=C_L(A)$.
\end{itemize}
\end{lemma}
\begin{proof}\begin{itemize} \item[(i)] We have
$$\begin{array}{lll}
x \in I_L(A/B) & \Leftrightarrow & \exists \hspace{.1cm} a' \in A \hbox{ such that }[x,a]+B=[a',a]+B \hspace{.3cm} \forall a \in A \\
 & \Leftrightarrow & \exists \hspace{.1cm} a' \in A \hbox{ such that }[x-a',a]+B=B \hspace{.3cm} \forall a \in A \\
 & \Leftrightarrow & \exists \hspace{.1cm} a' \in A \hbox{ such that }[x-a',a] \in B \hspace{.3cm} \forall a \in A \\
 & \Leftrightarrow & \exists \hspace{.1cm} a' \in A \hbox{ such that } x-a' \in C_L(A/B) \hspace{.1cm}  \\
 & \Leftrightarrow & x \in A+C_L(A/B) 
\end{array}$$
\item[(ii)] For $x \in I_L(A)$ let $a_x \in A$ be such that $x.a=a_xa$ for all $a \in A$. Define $\theta : I_L(A) \rightarrow A/Z(A)$ by $\theta(x)=a_x+Z(A)$. Then it is straightforward to check that $\theta$ is well-defined and is a homomorphism. Moreover, Ker$(\theta)=C_L(A)$, whence the result.
\item[(iii)] This is straightforward.
\end{itemize}
\end{proof}

Let\ $A,\ B$ be two $L$-algebras. If $A$ and $B$ are $L$-equivalent, then it is clear
from Proposition \ref{p:1} that $I_{L}( A) =I_{L}( B) .$

\begin{proposition}\label{p:3}
Let\ $L$ be a Lie algebra and let $F_{1},\ F_{2}\in C\mathcal{F}(L).$ Then the
following assertions are equivalent:
\begin{itemize}
\item[(i)] $F_{1}\sim _{L}F_{2};$
\item[(ii)] $ F_{1}$ and $F_{2}$ are $L$-connected;
\item[(iii)] either $F_{1}\cong _{L}F_{2}$ or there exist $E_{i}\in
C\mathcal{F}(L)$ such that $F_{i}\cong _{L}E_{i}$ for $ i=1,2$ , and the
$E_{i}$'s have a common complement in $L,$ which is a maximal subalgebra of $L;$ and
\item[(iv)] either $F_{1}\cong _{L}F_{2}$ or there exist $E_{i}\in C\mathcal{F}(L)$ such that $F_{i}\cong _{L}E_{i}$ for $i=1,2$, and the $E_{i}$'s have a common complement in $L.$
\end{itemize}
\end{proposition}

\begin{proof}
From \cite{[11]} we know that two abelian chief factors are $L$-equivalent if and only if they are $L$-isomorphic, and if and only if they are $L$-connected. Moreover, a complement $U$ of an abelian chief factor $A/B$ is a maximal
subalgebra and $L/U_{L}$ is primitive of type $1$ with $Soc(L/U_{L})=C/U_{L}$ and 
$C/U_{L}\cong _{L}A/B$, by \cite[Remarks following Proposition 2.5]{[11]}. So we
may assume that the chief factors are nonabelian and not $L$-isomorphic. Let $F_{1}=A/B$ and $F_{2}=D/E,$ where $A,B,C,D$ are ideals of $L.$
\begin{description}
\item[(i) $\Rightarrow$ (ii)] Put $X=C_{L}( A/B) 
$ and $Y=C_{L}( D/E) .$ Since $F_{1}$ and $F_{2}$ are
nonabelian we have that $X\neq Y,$ by \cite[Theorem 2.1]{[11]}
Also, since $F_{1}\sim _{L}F_{2},$ we have that $I_{L}( A/B) =I_{L}( D/E) :=I$.
Then $I=A+X=D+Y$, by Lemma \ref{l:1}. Also,
\[
\frac{X+Y}{X} \subseteq \frac{A+X}{X} \cong_L \frac{A}{B} \hbox{ and } \frac{X+Y}{Y} \subseteq \frac{D+Y}{Y} \cong_L \frac{D}{E},
\]
So $I=X+Y$, since $X \neq Y$ and
\[
X/X\cap Y\cong _{L}I/Y\cong _{L}D/E \hbox{ and } Y/Y\cap X\cong _{L}I/X\cong
_{L}A/B.
\]
It thus suffices to show that $L/X \cap Y$ is primitive of type $3$. Without loss of generality we can assume that $X\cap Y=0$. Then $C_{L}(X)=Y\text{ and }C_{L}\left( Y\right) =X$. Moreover, since $\sim_L$ is an equivalence relation, we have $Y\sim _{L}X.$ Thus  there is a $1$-cocyle $\alpha
\in Z^{1}\left( L,Y\right)$ and an $L$-isomorphism,\ $\phi :Y\rightarrow
X_{\alpha \text{ }}$, by Proposition \ref{p:1}. We also have that $U=Ker\left( \alpha \right) $ complements $X$ in $
L$, as in the proof of Proposition \ref{p:2}.  Now let $y\in Y$ and $u \in Y \cap U$. Then
\[
[u,\phi (y)] = 0 \hbox{ since } \phi(y) \in X=C_L(Y) \hbox{ and } u \in Y, \hbox{ and }
\]
\[
[\alpha(u),\phi (y)] = 0 \hbox{ since } \alpha(u)=0.
\]
But also,
\begin{equation*}
\phi \left( [u,y]\right) =u\odot \phi \left( y\right) =\left[ \alpha \left(
u\right) ,\phi \left( y\right) \right] +[u,\phi \left( y\right)] = 0
\end{equation*}%
whence, $[u,y]=0$, since $\phi$ is injective. It follows that $u\in C_{L}\left( Y\right) \cap Y =X \cap Y=0$ and so $Y \cap U=0$. Thus $U$ is a maximal subalgebra of $L$ with trivial core and $F_{1}\ $and\ $F_{2}$ are $L$-connected.
\item[(ii) $\Rightarrow$ (iii)] This follows immediately from the definition.
\item[(iii) $\Rightarrow$ (iv)] This is trivial.
\item[(iv) $\Rightarrow$ (i)]  If, $F_{1}\cong _{L}F_{2}$ then it is clear that $F_{1}\sim
_{L}F_{2}.$ So suppose that there exist $E_{i}\in C\mathcal{F}(L)$ such
that $F_{i}\cong _{L}E_{i}\ \left( i=1,2\right) ,$ and $E_{i}$'s have a
common complement in $L.$  Assume that the subalgebra $U$ of $L$ complements both $A/B$ and $D/E$ where
$A/B\cong _{L}F_{1}$ and $D/E\cong _{L}F_{2}.$ So $U$ also complements $(U_{L}+A)/U_{L}$ and $(U_{L}+ D)/U_{L}$.
Let 
\[
\phi : \frac{U_L+A}{U_L} \rightarrow \frac{U_L+ D}{U_L} \hspace{.2cm}
\hbox{ and } \hspace{.2cm} \beta : L \rightarrow \frac{U_L+ A}{U_L}
\] be given by $\phi
( a+U_{L}) =d+U_{L}$ if $a\in A$, $d\in D$ and $a+d \in U$, and $\beta(x) =a+U_{L}$ if $x\in L,$ $a\in A$ and $x+a\in U$. Then it is straightforward to check that $\beta \in
Z^{1}( L,(U_{L}+ A)/U_{L}) ,$ and that $\phi $ is an $L$-isomorphism. Thus $(U_{L}+ A)/U_{L}\sim _{L}(U_{L}+ D)/U_{L}$ and $A/B\sim _{L}D/E.$ This completes the proof.
\end{description}
\end{proof}

Now we will give a definition for the $A$-crown of $L$, which is a generalization of a
concept introduced by Towers in \cite{[11]}.
\par

Let $A$ be an irreducible $L$-algebra (that is, $A$ is a Lie algebra and an irreducible $L$-module). Put $I=I_{L}(A) .$ We set
\[
D_{L}( A) =\cap \{ R\mid R\subseteq I, R \hbox{ is an ideal of } L, A\sim _{L}I/R \hbox{ and }  I/R\text{ is non-Frattini}\}
\]
and
\[
E_{L}( A) =\{ x\in L \mid \alpha( x) =0 \hspace{.2cm} \text{ for all } \alpha \in Z^{1}( L,A)\}.
\]

Obviously, if $A\sim _{L}B,$ then $D_{L}( A) =D_{L}(B) $ and $E_{L}( A) =E_{L}( B) .$ The quotient, $I_{L}( A) /D_{L}( A) $ is then called the {\em $A$-crown} of $L.$

In \cite{[11]} the {\em crown} of a supplemented chief factor $A/B$ of $L$ was defined to be $C/R$, where 
$$C=A+C_L(A/B)$$ 
and 
$$R=\cap \{M_L \mid M  \in {\mathcal J}\},$$
where ${\mathcal J}$ is the set of all maximal subalgebras which supplement a chief factor $L$-connected to $A/B$. Clearly $C=I_L(A/B)$ and $R=D_L(A/B)$ where $A/B$ is considered to be an $L$-algebra in the natural way.

Let $A$ be an $L$-algebra. Then the set of 1-coboundaries, $Z^1(L,A)$ and the 1-dimensional cohomology space, $H^1(L,A)$, are defined in the usual way (see, for example \cite{CE}). We put $A^{L}:= \{a \in A \mid L.a=0\}$. Then $A^L= H^{0}( L,A)$. Let $N$ be an ideal of $L$. Then $A$ is, by restriction, an $N$-algebra,
and $Z^{1}( N,A)$,  $H^{1}( N,A) $  become $L$-modules. Moreover, we have the following inflation-restriction exact sequences.

\begin{lemma}\label{l:2}
Let%
\begin{equation*}
N\rightarrowtail L\twoheadrightarrow L/N
\end{equation*}%
be a short exact sequence of Lie algebras, where $N$ is an ideal of $L$ and the arrows are the canonical inclusion and projection. If $A$ is an $L$-algebra, we have the following exact sequences:
\begin{equation*}
0\longrightarrow Z^{1}( L/N,A^{N}) \overset{\text{inf}}{%
\longrightarrow }Z^{1}( L,A) \overset{\text{res}}{\longrightarrow 
}Z^{1}( N,A)
\end{equation*}
\begin{equation*}
0\longrightarrow H^{1}( L/N,A^{N}) \overset{\text{inf}}{%
\longrightarrow }H^{1}( L,A) \overset{\text{res}}{\longrightarrow 
}H^{1}( N,A) ^{L/N}
\end{equation*}
where inf and res denote the corresponding inflation and restriction maps.
\end{lemma}

\begin{theorem}\label{t:2}
Let $A$ be an irreducible $L$-algebra and let $N$ be an ideal of $L$ with $N\subseteq C_{L}\left( A\right) .$ Then the following are equivalent:
\end{theorem}

$
\begin{array}{ccc}
(1) \ N\subseteq E_{L}( A) & (2) \
Z^{1}( L,A) =Z^{1}( L/N,A) & (3) \
H^{1}( L,A) =H^{1}( L/N,A)
\end{array}
$

\begin{proof}
This follows from the above lemma. Note that the
inflation is bijective if and only if the restriction is null and that is
equivalent to $N\subseteq Ker( \alpha)$ for all $\alpha \in
Z^{1}( L,A) .$
\end{proof}

The analogue of the following result for groups was proved using cohomology theory. Here we give a more direct proof for the Lie algebra case.

\begin{theorem}
If\textbf{\ }$A$ is an abelian irreducible $L$-algebra, then $E_{L}(A) =D_{L}( A) .$
\end{theorem}
\begin{proof}  Put $I=I_L(A)=C_L(A)$, by Lemma \ref{l:1}. Let $\alpha \in Z^1(L,A)$.  First note that $\alpha|_I$ is an $L$-homomorphism from $I$ into $A$, since
\[ \alpha([x,y]) = x.\alpha(y)-y.\alpha(x)+\alpha(x)\alpha(y)=0 \text{ for all } x,y\in I, \text{and}
\]
\[ \alpha([x,i])= x.\alpha(i)-i.\alpha(x)+\alpha(x)\alpha(i)= x.\alpha(i) \text{ for all } x\in L, i\in I.
\]
It follows that $\alpha(I)$ is an $L$-submodule of $A$, and so $\alpha(I)=0$ or $A$, by the irreducibility of $A$. The former implies that $D_L(A) \subseteq I \subseteq \ker(\alpha)$. So suppose that $\alpha(I)=A$. Then $I/I\cap \ker(\alpha)\cong_L A$. Moreover, 
\begin{equation*}
\begin{array}{ll}
 \dim (I+\ker(\alpha)) & = \dim I+\dim \ker(\alpha)-\dim I\cap \ker(\alpha) \\ 
  & =\dim A+\dim \ker(\alpha)  \\
 & =\dim \text{im}(\alpha)+\dim \ker(\alpha)  \\
 & = \dim L.
\end{array} 
\end{equation*}
It follows that $L=I+\ker(\alpha)$, and $I/I\cap \ker(\alpha)$ is complemented by $\ker(\alpha)$ (which is a subalgebra of $L$) and so is non-Frattini. Hence $D_L(A) \subseteq I\cap \ker(\alpha)$.
\par

Thus, in either case, $D_L \subseteq \ker(\alpha)$, and $D_L(A) \subseteq E_L(A)$.
\par

Finally suppose that there exists $x \in E_L(A)$ such that $x \notin D_L(A)$. Then there exists $R \in D_L(A)$ such that $x \notin R$ but $x \in \ker(\alpha)$ for all $\alpha \in Z^1(L,A)$. Since $I/R$ is non-Frattini, there is a maximal subalgebra $M$ of $L$ such that $L=I+M$ and $I \cap M=R$.  Now there is a cocycle $\beta \in Z^1(L,B)$ and an $L$-isomorphism $\phi$ from $I/R$ onto $A_\beta$, by Proposition \ref{p:1}. Moreover $A_\beta =A$, since $A$ is abelian. So define $\alpha:L \rightarrow A$ by $\alpha(m)=0, \alpha(i)=\phi(i+R)$. Then it is straightforward to check that $\alpha \in Z^1(L,A)$ and that $M=\ker(\alpha)$. Furthermore, $x \in I \cap M=R$, contradiction. Hence $E_L(A) \subseteq D_L(A)$ and equality results.
\end{proof}

In the rest of this section we investigate the case where $A$ is nonabelian.

Recall that, if $A$ is an $L$-algebra, then $\alpha :L\rightarrow A$ is a $1$%
-cocyle if and only if $\alpha ^{\ast }:L\rightarrow A\rtimes L$ given by; 
\begin{equation*}
\alpha ^{\ast }( x) = (\alpha ( x),x)
\end{equation*}%
is a homomorphism and $\alpha \longmapsto \alpha ^{\ast }( L) $
defines a bijection between $Z^{1}( L,A) $ and the set of all
complements of $A$ in $A\rtimes L$. Then%
\begin{equation*}
\ker (\alpha) =\alpha ^{\ast }( L) \cap L
\end{equation*}%
We can give the following characterization:

\begin{theorem}\label{t:3}
Let $A$ be a nonabelian irreducible $L$-algebra. Then;
\begin{equation*}
E_L( A)_{A\rtimes L} =\cap \{ C_{L}( B) \mid B\sim
_{L} A\} .
\end{equation*}
\end{theorem}

\begin{proof}
By Proposition \ref{p:1} we have that%
\begin{equation*}
\cap \{ C_{L}( B) \mid B\sim _{L}A\}
=\cap \{ C_{L}( A_{\alpha }) \mid \alpha \in
Z^{1}( L,A)\} .
\end{equation*}%
Consider the semi-direct sum $A\rtimes L.$ From the remark above this
theorem, we have immediately that
\begin{equation*}
E_{L}( A) =\cap \{ H\mid H\text{ is a complement of }A%
\text{ in }A\rtimes L\} .
\end{equation*}%
In particular $E_{L}( A)_{A\rtimes L}$ is an ideal of $A\rtimes L$ and $
E_{L}( A) \cap A=0.$ As $E_{L}( A) \subseteq L,$ we
have that $E_{L}( A)_{A\rtimes L} \subseteq C_{L}( A) .$ On the
other hand, if $\alpha \in Z^{1}( L,A) $ and $x\in \ker(
\alpha ) ,$ then $x\in C_{L}( A) $ if and only if $x\in
C_{L}( A_{\alpha }) .$ So we have that
\begin{equation*}
E_{L}( A)_{A\rtimes L} \subseteq \cap \{ C_{L}( A_{\alpha })
\mid \alpha \in Z^{1}( L,A)\} .
\end{equation*}

Suppose now that $x\in \cap \{ C_{L}( A_{\alpha }) \mid
\alpha \in Z^{1}( L,A) \} .$ Then, for all $\alpha
\in Z^{1}\left( L,A\right) $ and for all $a\in A$ we have
\begin{equation*}
0=x\odot a=\alpha(x)a+x.a.
\end{equation*}%
Putting $\alpha =0$ we obtain that $x.a=0$ for all $a\in A.$ Thus, $\alpha(x)a=0$ for all $a\in A$, and so $\alpha ( x) \in Z(A) =0$ as $A$ is irreducible and nonabelian. Hence, $x\in E_{L}( A) .$ The reverse inclusion follows.
\end{proof}

\begin{lemma}\label{l:3}
Let $A$ be an irreducible $L$-algebra such that $C_{L}( A)
\subset I_{L}( A) .$ Then,
\begin{equation*}
D_{L}( A) \subseteq C_{L}( A) \Longleftrightarrow
I_{L}( A) /C_{L}( A) \cong _{L}A
\end{equation*}
\end{lemma}

\begin{proof}
Put $I:=I_{L}( A) ,$ etc. If $I/C\cong _{L} A$, then it is not Frattini, since it is
nonabelian. It follows from the definition of $D_{L}(A)$ that $ D_{L}( A) \subseteq C_{L}( A) .$ 
\par 

Suppose now that $D_{L}( A) \subseteq C_{L}( A) .$ Then $A$ is nonabelian, so
$I \neq C_{L}( A)$. Moreover, $I/D$ is completely reducible (as in \cite[Theorem 3.2]{[11]}), so $I_{L}( A) /C_{L}( A) \cong _{L}A.$
\end{proof}

\begin{corollary}\label{c:1}
Let $A$ be a nonabelian irreducible $L$-algebra such that%
\begin{equation*}
\{ B\in C\mathcal{F}( L) \mid B\sim _{L}A\} \neq \emptyset.
\end{equation*}%
Then
\begin{equation*}
D_{L}( A) =\cap \{ C_{L}( B) \mid B\sim
_{L}A,\ B\in C\mathcal{F}( L)\}
\end{equation*}
\end{corollary}
\bigskip

Let $A$ be a nonabelian irreducible $L$-algebra. We set
\[
J_{L}( A) =\cap \{ C_{L}( B) \mid B\sim
_{L}A,\ B\ncong _{L}F,\ F\in C\mathcal{F}( L)\}
\]

if $\{ B\mid B\sim _{L}A,\ B\ncong _{L}F,\ F\in C\mathcal{F}
( L)\} \neq \varnothing $ and we put $J_{L}(
A) =I_{L}( A) $, otherwise.

\begin{proposition}\label{p:4}
Let $A$ be a nonabelian irreducible $L$-algebra. Then
\begin{equation*}
I_{L}( A) =J_{L}( A)+D_{L}( A)
\end{equation*}%
and%
\begin{equation*}
J_{L}( A) \cap D_{L}( A) =E_{L}( A)_{A\rtimes L}
\end{equation*}
\end{proposition}

\begin{proof}
It is clear that $J_{L}( A)
\cap D_{L}( A) =E_{L}( A)_{A\rtimes L} .$ Let $B\sim _{L}A$ be such
that $B\ncong _{L}F$ if $\ F\in C\mathcal{F}( L) $, and put $S:=C_{L}( B)$, $I:=I_{L}( A) ,$ etc. Then there is a $1$-cocyle $\alpha \in Z^1(L,B)$ and an $L$-isomorphism $\phi : A \rightarrow B_{\alpha}$, by Proposition \ref{p:1}. Let $x \in S$ and $a \in A$. Then 
\[ 
\phi (x.a)=x\odot \phi (a)=\alpha(x) \phi (a)+x.\phi (a)=\alpha (x) \phi (a),
\]
so $x.a=\phi^{-1}\alpha(x)a$. It follows that $x\in I$ and $S \subseteq I$.  Now, $Z(B)=0$, since $B$ is nonabelian and an irreducible $L$-algebra, so there is an monomorphism $\theta : I/S=I_L(B)/S \rightarrow B$, by Lemma \ref{l:1}. Moreover, $\theta$ cannot be surjective, since $I/S\in C\mathcal{F}( L) $. Hence $I/S$ is isomorphic to a proper subalgebra of $B.$
\par

Now $D+S$ is
an ideal of $I$ and $I/D$ is completely reducible $L$-algebra, with
its irreducible components $L$-equivalent to $A$  (as in \cite[Theorem 3.2]{[11]}), and thus to $B$. If $A_i/D$ is an irreducible component of $I/D$, then $A_i \subseteq S$, as in Proposition \ref{p:2}. It follows that $D+S=I.$

Suppose that $D+J\subset I.$ Let $I/R$ be a chief factor of $L$ such
that $D+J\subseteq R.$ Then, $I/R\cong _{L}A$ because $D\subseteq R.$
As $J\subseteq R$, there exists $B\sim _{L}A$ with $B\ncong _{L}F$ if$\ F\in C%
\mathcal{F}( L) $, such that $I/C_{L}( B) $ has a
factor isomorphic to $I/R,$ contradicting the fact that $\dim (I/C_{L}(
B)  <\dim A .$
\end{proof}

\section{\textbf{On Complemented Chief Factors}}

Let $L$ be a Lie algebra. We say that a chief factor of $L$ is a {\em $c$-factor} if it is complemented in $L$ by a subalgebra, and that it is an {\em $m$-factor} if it is complemented by a maximal subalgebra of $L;$ otherwise we
say that it is a {\em $c'$-factor}, respectively an {\em $m'$-factor}

Observe that, an abelian chief factor is an $m$-factor (
respectively an $m'$-factor) if and only if it is a $c$-factor (respectively, a 
\textit{Frattini} chief factor).
\par

Let $A/B$ and $C/D$ be chief factors of $L$. We write $A/B \searrow C/D$ if $A=B+C$ and $B \cap C=D$.  If $A/B \searrow C/D$, $A/B$ is a Frattini chief factor and $C/D$ is supplemented by a maximal subalgebra of $L$, then we call this situation an {\em $m$-crossing}, and denote it by $[A/B \searrow C/D]$. 
\par

We say that two chief factors $A/B$ and $C/D$ of $L$ are {\em $m$-related} if one of the following holds.
\begin{itemize}
\item[1.] There is a supplemented chief factor $R/S$ such that $A/B \swarrow R/S \searrow C/D$.
\item[2.] There is an $m$-crossing $[U/V \searrow W/X]$ such that $A/B \swarrow V/X$ and $W/X \searrow C/D$.
\item[3.] There is a Frattini chief factor $Y/Z$ such that $A/B \searrow Y/Z \swarrow C/D$.
\item[4.] There is an $m$-crossing $[U/V \searrow W/X]$ such that $A/B \searrow U/V$ and $U/W \swarrow C/D$.
\end{itemize}
Then we have the following result.

\begin{proposition}\label{p:5}
Let $L$ be a Lie algebra over any field, let $H$ and $K$ be ideals of $L$ with $H \subseteq K$, and let
\begin{equation*}
H=X_{0}<X_{1}<X_{2}<...<X_{n}=K
\end{equation*}%
and%
\begin{equation*}
H=Y_{0}<Y_{1}<Y_{2}<...<Y_{m}=K
\end{equation*}%
be two sections of chief series of $L$ between $H$ and $K$. Then $n=m$ and there exists a unique permutation $\pi $
in $S_{n}$ such that $X_i/X_{i-1}$ and $Y_{\pi(i)}/Y_{\pi(i)-1}$ are $m$-related. In particular,

$\left( i\right) \ X_{i}/X_{i-1}\sim _{L}Y_{\pi \left( i\right) }/Y_{\pi
\left( i\right) -1}$

$\left( ii\right) \ X_{i}/X_{i-1}$ and $Y_{\pi \left( i\right) }/Y_{\pi
\left( i\right) -1}$ are simultaneously $m$\textit{-factors} or $m^{%
%TCIMACRO{\U{b4}}%
%BeginExpansion
{\acute{}}%
%EndExpansion
}$\textit{-factors.}

$\left( iii\right) \ $If$\ X_{i}/X_{i-1}$ and $Y_{\pi \left( i\right)
}/Y_{\pi \left( i\right) -1}$ are $m$\textit{-factors}, then they have a
maximal subalgebra of $L$ as a common complement.
\end{proposition}
\begin{proof}
This follows from \cite[Theorems 2.9 and 2.7]{[12]}.
\end{proof}

In particular, the number of $m$\textit{-factors} in any chief series of $L$
are the same. But this is no longer true for $c$\textit{-factors,} in spite
of the equivalence between $\left( 3\right) $ and $\left( 4\right) $ in
Proposition \ref{p:3}, as we shall see in a later example.
\par

If $S$ is a subalgebra of $L$, the {\em normalizer of $S$ in $L$} is defined as
\begin{equation*}
N_{L}(S)=\left\{ x\in L\left\vert \left[ x,S\right] \subseteq S\right.
\right\}.
\end{equation*}

\begin{lemma}\label{l:4}
Assume that $B^*/B$ is a $c'$-factor and that $A^*/A$ is a $c$-factor of $L$, both of which are nonabelian and such
that $B^*/B\searrow A^*/A$. Let $%
I=I_{L}\left( A^{\ast }/A\right) $ and $C=C_{L}\left( A^{\ast
}/A\right) .$ Then
\begin{itemize}
\item[(i)] $I/C \searrow B^*/B$ and $I/C$ is a $c'$-factor;
\item[(ii)] there exists an ideal $X$ of $L$ with $X \subseteq I$ such that $X/N \searrow A^*/A$, where $N=X\cap C$, $I/C \searrow X/N$ and $X/N$ is a $c$-factor;
\item[(iii)]there exists a supplement $F$ of $I/C$ in $L$ such that $L/N$ is isomorphic to the natural semi-direct sum of $I/C$ by $F/C$; and
\item[(iv)] $L/C$ is a primitive Lie algebra of type $2$ and $Soc(L/C)=I/C$.
\end{itemize} 
\end{lemma}

\begin{proof}
We have that $B^*/B\searrow A^*/A$, so $A^*+B=B^*$ and $A^* \cap B=A$. Also, $[B,A^*]+A=A$ or $A^*$. But the latter implies that $A^* \subseteq A+B\cap A^*=A$, a contradiction, so $[B,A^*] \subseteq A$; that is, $B \subseteq C$.
Hence $B^*+C=A^*+B+C=A^*+C=I$, by Lemma \ref{l:1}, and $B^*\cap C=(A^*+B)\cap C=B+A^*\cap C=B+A=B+A^*\cap B=B$. Thus $I/C\searrow B^{\ast }/B$. Suppose that $I/C$ is a $c$-factor of $L$. Then there is a subalgebra $U$ of $L$ such that $L=I+U$ and $I\cap U= C$. But now $L=B^*+C+U=B^*+U$ and $B^*\cap U=B^*\cap I\cap U=B^*\cap C=B$, so $B^*/B$ is a $c$-factor, a contradiction. Thus, $I/C$ is a $c'$-factor of $L$ and we have (i).
\par 

Let
\[ A =A_0 < A_1 < \ldots < A_n = C
\]
be part of a chief series of $L$ between $A$ and $C$. Then
\[ A^*=A^*+A_0 < A^*+A_1 < \ldots < A^*+A_n=I
\]
is part of a chief series of $L$ between $A^*$ and $I$. Suppose that $(A^*+A_i)/A_i$ is a $c$-factor for some $1 \leq i \leq n-1$. Then there is a subalgebra $U$ such that $L=A^*+A_i+U$ and $(A^*+A_i)\cap U=A_i$. Then $A_i \subseteq U$ so $L=A^*+U=A^*+A_{i-1}+U$ and $(A^*+A_{i-1})\cap U=A^*\cap U+A_{i-1}=A_{i-1}$, since $A^*\cap U \subseteq A^*\cap A_i=A$. Thus $(A^*+A_{i-1})/A_{i-1}$ is a $c$-factor. It follows that $(A^*+A_k)/A_k$ is a $c$-factor and $(A^*+A_{k+1})A_{k+1}$ is a $c'$-factor for some $0 \leq k \leq n-1$, since $A^*/A$ is a $c$-factor and $I/C$ is a $c'$-factor. Put $N=A_k$, $X=A^*+A_k$, $Y=A_{k+1}$ and $M=A^*+A_{k+1}$. Then it is straightforward to check that
\begin{equation*}
I/C\searrow M/Y\searrow X/N\searrow A^{\ast }/A
\end{equation*}
where $M/Y$ is a $c'$-factor and $X/N$ is a $c$-factor and we have (ii). 
\par

Without loss of generality we may assume that $N=0$. Let $U$ be a complement of $X$ in $L$, so $L=X+U$ and $X\cap U=0$, and consider, $K=U\cap C$. Then $[X,C] \subseteq X\cap C=N=0$, since $I/C \searrow X/N$, so $K$ is an ideal of $L$. We have $U+(K+X) =L$ and $U\cap ( K+X) =K+( U\cap X)=K$. It follows that $(K+X)/K$ is a $c$-factor of $L$. Also $K+X+C=X+C=A^*+A_k+C=A^*+C=I$ and $(K+X)\cap C=K+X\cap C=K$, so $I/C\searrow (K+X)/K$ and we may assume that $K=0$.

Observe that the map
\begin{equation*}
u+x\longrightarrow (x+C, u+C)
\end{equation*}%
where $u\in U$ and $x\in X$ defines an epimorphism between $L=X+U$ and the
natural semidirect sum $ I/C\rtimes((U+C)/C)$. Furthermore, it is easy to check that the kernel of this map is $N$, so putting $F=U+C$ proves (iii).
\par

We have
\[ \frac{I}{C}=\frac{A^*+C}{C}\cong \frac{A^*}{A^*\cap C}=\frac{A^*}{A},
\]
and $[U_L,I] \subseteq C$, whence $[U_L,A^*] \subseteq A$ and $U_L \subseteq C$. This establishes (iv).
\end{proof}

We can construct an example of the situation in Lemma \ref{l:4} as follows.
\begin{example} Let $L_{0}$ be a primitive Lie algebra of type $2$ with $Soc( L_{0})
=X_{0}$, where $X_{0}$ is not complemented in $L_{0}$, and let $U_0$ be a supplement to $X_0$ in $L_0$. So, for example, we could take $L_0=sl(2)\otimes {\mathcal O}_m+1\otimes {\mathcal D}$, $X_0=sl(2)\otimes {\mathcal O}_m$, $U_0=(Fe_0+Fe_1)\otimes {\mathcal O}_m+1\otimes {\mathcal D}$ where ${\mathcal O}_m$ is the truncated polynomial algebra in $m$ indeterminates, ${\mathcal D}$ is a solvable subalgebra of Der(${\mathcal O}_m$), ${\mathcal O}_m$ has no ${\mathcal D}$-invariant ideals, and the ground field is algebraically closed of characteristic $p>5$ (see \cite[Theorem 6.4]{pasha}).

 Put $Y_0=U_0\cap X_0$. Then $X_0$ is a $U_0$-algebra and so we can form the semi-direct sum
\[ L=X_0 \rtimes U_0= \{(x,u) \mid x \in X_0, u \in U_0 \}.
\]
Put $X=\{ (x,0) \mid x \in X_0\}$, $U=\{ (0,u) \mid u \in U_0\}$. Then $L=X+U$, $X\cap U=0$ and $X$ is an ideal of $L$.
\par

Now let $B=\{(0,y) \mid y \in Y_0\}$, $W=\{  (y,0)\mid y \in Y_0\}$. Putting $C=\{(y,-y) \mid y \in Y_0\}$ and $I=X+C$, we have that
\begin{itemize}
\item[(1)] $C$ is an ideal of $L$, $X\cap C=0$, $I=X+B$ and $B=U\cap I$;
\item[(2)] $X\cap (B+C)=W$, $B+W=B+C$, $W$ is an ideal of $C+U$ and $[W,C]=0$;
\item[(3)] $W \cong_U B \cong_U C$.
\end{itemize}
Consider the following chief series of $L$.
\[  0 < C < I < \ldots < L  \hbox{ and } 0 < X < I < \ldots < L.
\]
We have the situation of Lemma \ref{l:4} with $N=0$. Then $I/C$ is a $c'$-factor and $X/0$ is a $c$-factor as in the lemma. Suppose that $C$ is complemented in $L$, so there is a maximal subalgebra $M$ of $L$ such that $L=C\dot{+} M$. Then $[C,X]=0$ so $X\cap M$ is an ideal of $L$. But $X$ is a minimal ideal of $L$, so $X\cap M=0$ or $X \subseteq M$. The former implies that $\dim(X+M)=\dim X+\dim M>\dim Y_0+\dim M=\dim C+\dim M= \dim L$, which is impossible. The latter implies that $L=I+M$ and $I\cap M=(C+X)\cap M=C\cap M+X=X$, so $M$ is a complement for $I/X$. If the chief factors between $I$ and $L$ are the same in each series then the second series has one more complemented chief factor than the first.  
\end{example}
Then we have the following proposition.

\begin{proposition}\label{p:6}
Suppose that, in the situation of Proposition \ref{p:5}, $X_{i}/X_{i-1}$ and $Y_{\pi(
i) }/Y_{\pi( i) -1}$ are $m'$-factors. Then we have:
\begin{itemize}
\item[(a)] both factors are $c'$-factors; or
\item[(b)] both factors are nonabelian $c$-factors; or
\item[(c)] both factors are nonabelian, one of them is a $c$-factor, the other one is a $c'$-factor, and there exist ideals $I,\ C,\ X,\ $and $N$ of $L$ of
satisfying $(i) $-$(iv) $ of Lemma \ref{l:4}.
\end{itemize}
\end{proposition}

\begin{proof}
Assume that $X_{i}/X_{i-1}$ is a $c'$-factor and that $Y_{\pi( i) }/Y_{\pi( i) -1}$
is a $c$-factor. As these two chief factors are $m$-related, one of the following situations arises.
\begin{itemize}
\item[1.]  There is a supplemented chief factor $R/S$ such that 
$$X_i/X_{i-1} \swarrow R/S \searrow Y_{\pi( i) }/Y_{\pi( i) -1}.$$
 Since $X_i/X_{i-1}$ is a $c'$-factor, so is $R/S$. We are thus in the situation of Lemma \ref{l:4} with $B^*=R$, $B=S$, $A^*= Y_{\pi(i)}$, $A^*=Y_{\pi(i)-1}$.
\item[2.] There is an $m$-crossing $[U/V \searrow W/X]$ such that 
$$X_i/X_{i-1} \swarrow V/X \text{ and } W/X \searrow Y_{\pi( i) }/Y_{\pi( i) -1}.$$  
Then \cite[Theorem 2.4]{[12]} implies that $[U/W \searrow V/X]$ is also an $m$-crossing.
\par

Suppose that $W/X$ is supplemented by the maximal subalgebra $M$ of $L$, so $L=W+M$ and $X \subseteq W\cap M$. Then $L=W+M=U+M$. If $V \subseteq M$ then $V \subseteq U\cap M$ and $U/V$ is supplemented by $M$, contradicting the fact that $U/V$ is Frattini. Hence $V \not \subseteq M$. It follows that $L=V+M$. Moreover, $X \subseteq  V\cap M \subseteq V$. As $V/X$ is a chief factor of $L$ we have $V\cap M=X$, and so $V/X$ is a $c$-factor. But then $X_i/X_{i-1}$ is a $c$-factor, which is a contradiction. Thus this case cannot occur.
\item[3.] There is a Frattini chief factor $Y/Z$ such that 
$$X_i/X_{i-1} \searrow Y/Z \swarrow Y_{\pi( i) }/Y_{\pi( i) -1}.$$ 
Since $Y_{\pi( i) }/Y_{\pi( i) -1}$ is a $c$-factor, so is $Y/Z$. But $Y/Z$ is Frattini, so this is impossible and this case cannot occur.
\item[4.] There is an $m$-crossing $[U/V \searrow W/X]$ such that 
$$X_i/X_{i-1} \searrow U/V \text{ and } U/W \swarrow Y_{\pi( i) }/Y_{\pi( i) -1}.$$ 
Then \cite[Theorem 2.4]{[12]} implies that $[U/W \searrow V/X]$ is also an $m$-crossing, so $U/W$ is a $c'$-factor, contradicting the fact that $Y_{\pi( i) }/Y_{\pi( i) -1}$ is a $c$-factor. Hence this case cannot occur either.
\end{itemize}
\end{proof}

Let $A$ be an irreducible $L$-algebra. We say that $A$ is of {\em $cc'$-type} in $L$ if there exist two chief series of $L$ in which case (c)  of Proposition \ref{p:6} holds with $A\sim
_{L}X_{i}/X_{i-1}$ (Clearly this forces $A$ to be nonabelian.)

\begin{proposition}\label{l:7}
Let $v$ be the number of equivalence classes of irreducible $L$%
-algebras of $cc'$-type. Then the number of complemented chief factors on two chief
series of $L$ differs by at most $v.$
\end{proposition}

\begin{proof}
A consequence of Proposition \ref{p:6} is that, on a chief series of $L$, for each
non-abelian crown there is at most one $m'$-factor. If the crown corresponds to a factor of $cc'$-type, this shows that on each chief series there is at most one $c'$-factor corresponding to the crown.
\end{proof}

\begin{theorem}
Let $A$ be a nonabelian irreducible $L$-algebra. Then $A$ is of $cc'$-type in $L$ if and only if
\begin{equation*}
E_{L}\left( A\right) \subset D_{L}\left( A\right) \subset I_{L}\left(
A\right)
\end{equation*}%
and $Soc( P) $ is a $c'$-factor of $P,$ where $P$ is the corresponding primitive
epimorphic image of $L.$
\end{theorem}

\begin{proof}
Put $E=E_{L}( A)$, $D=D_L(A)$ and $I=I_L(A)$ and suppose that $E\subset D\subset I.$ Then, since $D_L(A) \neq \emptyset$, there is an ideal $R$ of $L$ such that $I/R \sim_LA$. Also, $J_L(A) \neq I$, since otherwise $D=E$, by Proposition \ref{p:4}, so there is an ideal $B$ of $L$ with $B \sim_L A$ and $B \not \cong_L F$ if $F \in \mathcal{CF}(L)$. Put $H=C_L(B)$. Then $H \subseteq I_L(B)=I$, by Proposition \ref{p:1}, and $H \neq I$, by Lemma \ref{l:1} (iii), so $H \subset I$. Put $K=H\cap R.$ Then $H/K\cong_L (H+R)/R$. Moreover, if $H+R=R$ then $H=R$, $H \subset I=I_L(B)$ and $D_L(B)=D\subseteq R=H$, so $I_L(B)/C_L(B) \cong_L B$, by Lemma \ref{l:3}, contradicting the fact that $B \not \cong_L F$ if $F \in \mathcal{CF}(L)$. It follows that $H+R=I$ and $H/K \cong_L I/R$, whence $A\sim _{L}H/K.\ $By Proposition \ref{p:2}, $H/K$ is a $c$-factor of $L$ and $I/R$ is a $c'$-factor and so we have that $A$ is of $cc'$-type.

Conversely, if $A$ is of $cc'$-type, from the definition we obtain ideals $I,\ C,\ X,\ $and $N$ of $L$ and a
subalgebra $U$ of $L$ such that $I/C\sim _{L}A,\ I/C$ and $X/N\ $are $m'$-factors, $I/C$ is a $c'$-factor, $I/C\searrow X/N$, and $U$ complements $X/N$ in $L$ (using the same notation as Proposition \ref{p:6}).Note that $I=C+X$ and $C\cap X=N$, so $I/C \cong_L X/N$, whence $C_{L}(X/N)=C_{L}( I/C) =C$, by Lemma \ref{l:4} (iii). Now $[L,U\cap C]=[U+X,U\cap C] \subseteq U\cap C$ since $[X,C] \subseteq N \subseteq U\cap C$, so $U\cap C$ is an ideal of $L$. Also 
\[ \frac{X+U\cap C}{U\cap C} \searrow \frac{X}{N}.
\]
\par

As in the proof of Proposition \ref{p:2}, we obtain an $L$-algebra $B,$ $B\sim
_{L}A$ such that
\begin{equation*}
C_{L}( B) =X+C_{U}(X/N) =X+ U\cap
C.
\end{equation*}

Suppose now that there exists $F\in C\mathcal{F}( L) ,$ such that 
$F\cong _{L}B.$ Then $I_{L}( F) =I_{L}( B)
=I$ and $C_{L}( F) =C_{L}( B) =X+U\cap C,$ and so $F\cong
_{L}I/(X+U\cap C).$ It follows that $I/(X+U\cap C) \sim _{L}I/C,$ which is a primitive Lie
algebra of type $2$, by Lemma \ref{l:4} (iv). But
\[ \left. \frac{L}{X+U\cap C} \cong_L \frac{L}{U\cap C}\middle/ \frac{X+U\cap C}{U\cap C} \right.,
\] so $L/U\cap C$ is primitive of type $3$. It follows that $(X+U\cap C)/U\cap C$ is an $m$-factor, and hence so is $X/N$, by \cite[Lemma 2.1]{[12]}, a contradiction.

 Moreover we have that
\begin{equation*}
D_{L}\left( A\right) \subseteq C\subset I,\ J_{L}\left( A\right) \subseteq
X\subset I,
\end{equation*}
which completes the proof.
\end{proof}


\begin{thebibliography}{99}
\bibitem{[1]} $\ $D. W. BARNES, `On the Cohomology of Soluble Lie Algebras', 
\textit{Math. Zeitschr}. \textbf{101} (1967), 343-349.

\bibitem{[2]} R. E. BLOCK,`Determination of the Differentiably Simple Rings
with a Minimal Ideal', \textit{Ann. of Math. }\textbf{90} (1969), 433-459.

\bibitem{[3]} A. BALLESTER-BOLINCHES AND L. M. EZQUERRO, `Classes of Finite
Groups', \textit{Mathematics and its Applications} \textbf{584} (2006),
Springer.\smallskip

\bibitem{CE} C. CHEVALLEY and S.EILENBERG, `Cohomology theory of Lie groups and Lie algebras', \textit{Trans. Amer. Math. Soc.} \textbf{64} (1948), 85-124.

\bibitem{[4]} H. STRADE\ AND R. FARNSTEINER, `Modular Lie Algebras and Their
Representations', \textit{Marcel-Dekker, New York and Basel }1988.

\bibitem{[5]} J. P. LAFUENTE, `Nonabelian Crowns and Schunck Classes of
Finite Groups', \textit{Arch. Math}. \textbf{42}(1984), 32-39.

\bibitem{pasha} P. ZUSMANOVICH, `Deformations of $W_1(n) \otimes A$ and modular semisimple Lie algebras with a solvable maximal subalgebra', \textit{J. Algebra} \textbf{268}(2003), 603-635.

\bibitem{[6]} P. J. SERAL AND J. P. LAFUENTE, `On Complemented Nonabelian
Chief Factors of a Finite Group', \textit{Isr. J. of Math.,} \textbf{106},
(1998), 177-188.

\bibitem{[7]} D. A. TOWERS, `A Frattini Theory for Algebras', \textit{%
Proc. London Math. Soc.} (3) \textbf{27} (1973), 440-462.

\bibitem{[8]} D. A. TOWERS, `On Lie Algebras in which Modular Pairs of
Subalgebras are Permutable', \textit{J. Algebra }\textbf{68} (1981), 369-377.

\bibitem{[9]} D. A. TOWERS, `Complements of Intervals and Prefrattini
Subalgebras of Solvable Lie Algebras', \textit{Proc. Amer. Math. Soc.} 
\textbf{141} (2013), 1893-1901.

\bibitem{[10]} D. A. TOWERS, `On Conjugacy of Maximal Subalgebras \ of
Solvable Lie Algebras', \textit{Comm. Alg.} \textbf{42 (3)} (2014),
1350-1353.

\bibitem{[11]} D. A. TOWERS, `Maximal Subalgebras and Chief Factors of Lie
Algebras', {\em J. Pure Appl. Algebra} {\bf 220} (2016), 482--493..

\bibitem{[12]} D. A. TOWERS, `On Maximal Subalgebras and a Generalised Jordan-H\"older Theorem for Lie
Algebras', arxiv.org/abs/1509.06951.
\end{thebibliography}
\end{document}